\documentclass[10pt]{amsart}
\newtheorem{theorem}{Theorem}[section]
\newtheorem{lemma}[theorem]{Lemma}
\newtheorem{proposition}[theorem]{Proposition}
\newtheorem{corollary}[theorem]{Corollary}
 
\theoremstyle{definition}

\newtheorem*{example}{Example}

\numberwithin{equation}{section}

\begin{document}

\title{Fibonacci Expansions}
\author{Claudio Baiocchi}
\address{Accademia Nazionale dei Lincei, Palazzo Corsini, Via della Lungara 10, 00165 Roma, Italy}
\author{Vilmos Komornik}
\address{
D\'epartement de math\'ematique,
Universit\'e de Strasbourg,
7 rue Re\-n\'e Des\-car\-tes,
67084 Strasbourg Cedex, France}
\email{vilmos.komornik@math.unistra.fr}
\author{Paola Loreti}
\address{Sapienza Universit\`a di Roma,
Dipartimento di Scienze di Base
e App\-licate per l'Ingegne\-ria,
via A. Scarpa n. 16,
00161 Roma, Italy}
\email{paola.loreti@sbai.uniroma1.it}
\thanks{Claudio Baiocchi passed away on December 14, 2020.
His co-authors are grateful to his essential contributions to our collaboration.}
\thanks{The work of the second author was partially supported by the National Natural Science Foundation
of China (NSFC)  \#11871348.}

\begin{abstract}
Expansions in the Golden ratio base have been studied since a pioneering paper of R\'enyi more than sixty years ago.
We introduce closely related expansions of a new type, based on the Fibonacci sequence, and we show that in some sense they behave better.
\end{abstract}

\subjclass[2010]{Primary: 11A63, Secondary: 11B39}
\keywords{Non-integer base expansions, Kakeya sequences, Fibonacci sequence, Golden ratio}

\maketitle

\section{Introduction }\label{s1}

Expansions of the form
\begin{equation*}
x=\sum_{i=1}^{\infty}\frac{c_i}{q^i},\quad (c_i)\in\{0,1\}^{\mathbb N}
\end{equation*}
in real bases $q\in (1,2)$ have been first studied by R\'enyi \cite{Ren1957}.
Subsequently many works have been devoted to their connections to probability and ergodic  theory, combinatorics, symbolic dynamics, measure theory, topology and number theory; see, e.g., 
\cite{Gelfond1959,
C-Par1960,
KomLorPed2000,
Sid2003a,
C-BaiKom2007, 
DajDev2007,
KomLaiPed2011,
AkiKom2013,
Fen2016} 
and their references. 
Following the discovery of surprising uniqueness phenomena by Erd\H os et al. \cite{ErdHorJoo1991}, a rich  theory of unique expansions has been uncovered \cite{ErdJooKom1990,
C-DarKat1995,
ErdKom1998,
C-KomLor1998,
C-KomLor2007,
C-DevKom2009} 
.
There are still many open problems, for example concerning the number of possible  expansions of specific numbers in particular bases.

In the special case where $q=\varphi:=\frac{1+\sqrt{5}}{2}\approx 1.618$ is the Golden ratio, R\'enyi proved that the average distribution of the digits $0$ and $1$ is not the same, and he computed their frequencies.
In this paper we introduce the \emph{Fibonacci expansions} 
\begin{equation*}
x=\sum_{i=1}^{\infty}\frac{c_i}{F_i},\quad (c_i)\in\{0,1\}^{\mathbb N},
\end{equation*}
where the powers $\varphi^i$ are replaced by the Fibonacci numbers:
\begin{equation}\label{11}
F_1:=1,\quad F_2:=1,\quad\text{and}\quad F_{i+2}:=F_{i+1}+F_i,\quad i=1,2,\ldots .
\end{equation}
They are closely related to the expansions in base $\varphi$, because
\begin{equation}\label{12}
F_i=
\frac{1}{\sqrt{5}}\left(\varphi^i+\frac{(-1)^{i+1}}{\varphi^i}\right)
\end{equation}
for all $i$ by Binet's formula, whence $F_i$ is the nearest integer to 
${\varphi^i}/{\sqrt{5}}$.

The purpose of this work is to compare these two  expansions, and to study more general  \emph{Kakeya expansions}  of the form
\begin{equation}\label{13}
x=\sum_{i=1}^{\infty}{c_i}{p_i},\quad (c_i)\in\{0,1\}^{\mathbb N},
\end{equation}
where $(p_i)$ is a given \emph{Kakeya sequence}, i.e., a sequence of positive numbers satisfying the conditions $p_i\to 0$, and
\begin{equation}\label{14}
p_n\le \sum_{i=n+1}^{\infty}p_i\quad\text{for all}\quad n.
\end{equation}
We recall the following classical theorem:

\begin{theorem}[Kakeya \cite{Kak1914,Kak1915}]\label{t11}
If $(p_i)$ is a Kakeya sequence, then
a real number $x$ has an expansion of the form \eqref{13} if and only if $x\in\left[0,\sum_{i=1}^{\infty}p_i\right]$.
\end{theorem}
For example, $(q^{-i})$ is a Kakeya sequence for every $q\in(1,2]$, so that 
every $x\in [0,\frac{1}{q-1}]$ has an expansion in  base $q$. 
A similar result holds for Fibonacci expansions.
Setting 
\begin{equation}\label{15}
S:=\sum_{i=1}^{\infty}\frac{1}{F_i}\approx 3.360,
\end{equation}
we have the following result:

\begin{theorem}\label{t12}
A real number $x$ has an expansion in the Fibonacci base if and only if $x\in [0,S]$.
\end{theorem}

Next we will investigate the number of expansions. 
It is clear that the expansions of $0$ and $\frac{1}{q-1}$ are unique in every base $q\in(1,2]$: $c_i\equiv 0$ and $c_i\equiv 1$, respectively. 
Otherwise, $x$ may have several, even infinitely many expansions:

\begin{theorem}[Erd\H{o}s et al. \cite{ErdJooKom1990}]\label{t13}
If $q\in(1,\varphi)$, then every  $x\in(0,\frac{1}{q-1})$ has a continuum of expansions in base $q$.
\end{theorem}

By a different proof, we will extend Theorem \ref{t13} to a class of Kakeya expansions:

\begin{theorem}\label{t14}
Let $(p_i)$ be a sequence of positive real numbers, satisfying the following conditions: 
\begin{align}
&p_i\to 0;\label{16}\\
&p_n<\sum_{i=n+1}^\infty p_i\quad\text{for all}\quad n;\label{17}\\
&p_{n-1}<\sum_{i=n+1}^\infty p_i\quad\text{for infinitely many}\quad n;\label{18}\\
&p_n\le 2p_{n+1}\quad\text{for all sufficiently large}\quad n.\label{19}
\end{align}
Then every $0<x<S:=\sum_{i=1}^{\infty}p_i$ has a continuum of expansions of the form \eqref{13}.
\end{theorem}

The assumption $q\in(1,\varphi)$ of Theorem \ref{t13} is sharp: if $q=\varphi$, then for example $1$ has only countably many expansions by a theorem of Erd\H{o}s et al. \cite{ErdHorJoo1991}.
Applying Theorem \ref{t14} we will prove that the Fibonacci expansions behave better:

\begin{theorem}\label{t16}
Every $x\in (0,S)$ has a continuum of Fibonacci expansions.
\end{theorem}

For the reader's convenience we give a short  proof of Theorem \ref{t11} in Section \ref{s2}.
Then Theorems \ref{t12}, \ref{t14} and \ref{t16} are proved in Sections \ref{s3}, \ref{s4} and \ref{s5} , respectively.
Ate the end of Section \ref{s4} we  also deduce Theorem \ref{t13} from Theorem \ref{t14}.

\section{Proof of Theorem \ref{t11}}\label{s2}

First we prove Theorem \ref{t11}.
Given an arbitrary $x\in[0,S]$, we define
a function $f:\mathbb N\to\{1,2\}$ recursively as follows.
Set $s_1:=x$ and $s_2:=S-x$.
If $n\ge 1$ and $f(1),\ldots, f(n-1)$ have already been defined (no assumption if $n=1$), then we choose $j\in \{1,2\}$ such that
\begin{equation*}
p_n+\sum_{i<n, f(i)=j}p_i\le s_j,
\end{equation*}
and we define $f(n):=j$. 
Such a $j$ exists, for otherwise we would have
\begin{equation*}
2p_n+\sum_{i<n}p_i
=\sum_{j=1}^2\left(p_n+\sum_{i<n, f(i)=j}p_i\right)
>s_1+s_2
=S=\sum_{i=1}^{\infty}p_i,
\end{equation*}
contradicting the Kakeya property \eqref{14}. 

The sets $S_1:=\{i\in\mathbb N\ :\ f(i)=1\}$
and $S_2:=\{i\in\mathbb N\ :\ f(i)=2\}$ form a partition $\mathbb N$ such that 
\begin{equation*}
\sum_{i\in S_j}p_i\le s_j,\quad j=1, 2.
\end{equation*}
Both inequalities are in fact equalities because
\begin{equation*}
\sum_{i\in S_1}p_i
+\sum_{i\in S_2}p_i
=\sum_{i=1}^{\infty}p_i
=S=s_1+s_2.
\end{equation*}
In particular,
\begin{equation*}
\sum_{i\in S_1}p_i=x.
\end{equation*}

\section{Proof of Theorem \ref{t12}}\label{s3}

We need a lemma:

\begin{lemma}\label{l31}
We have $F_i\to +\infty$.
Furthermore,
\begin{equation*}
F_{n+1}\le 2F_n
\quad\text{and}\quad 
\frac{1}{F_n}<\sum_{i=n+1}^{\infty}\frac{1}{F_i}
\end{equation*}
for all $n$.
\end{lemma}

\begin{proof}
It follows from the definition that the Fibonacci sequence is a strictly increasing sequence of positive integers.
Therefore $F_i\to+\infty$.

Since $F_2<2F_1$ and $F_3=2F_2$, a trivial induction argument based on the identity $F_{i+2}=F_{i+1}+F_i$ shows that $F_{n+1}\le 2F_n$ for all $n\ge 1$, and equality holds only if $n=2$.
From this, again by induction, we have
$F_{n+j}\le 2^jF_n$ for all $n\ge 1$ and $j\ge 1$, and equality holds only if $n=2$ and $j=1$.
Indeed, we have
\begin{equation*}
F_{n+j}\le 2F_{n+j-1}
\le 4F_{n+j-2}
\le\cdots
\le 2^jF_n,
\end{equation*}
and at least one the inequalities is strict unless $n=2$ and $j=1$.
Therefore
\begin{equation*}
\sum_{j=1}^{\infty}\frac{1}{F_{n+j}}>\sum_{j=1}^{\infty}\frac{1}{2^jF_n}=\frac 1{F_n}
\end{equation*}
for every $n$.
\end{proof} 

\begin{proof}[Proof of Theorem \ref{t12}]
Since the Fibonacci numbers are positive, 
 the formula $p_i:=1/F_i$ defines a Kakeya sequence by Lemma \ref{l31}, and Theorem \ref{t12} follows from Theorem \ref{t11}.
\end{proof}

We may place Theorem \ref{t12} into a broader framework, by considering expansions of the form
\begin{equation}\label{31}
x=\sum_{i=1}^{\infty}\frac{c_i}{q^i(1+\varepsilon_i)},\quad (c_i)\in\{0,1\}^{\mathbb N}
\end{equation}
with some given real numbers $q\in (1,2)$ and $\varepsilon_i>-1$.
The following result shows that $(q^{-i})$, and even some perturbations of  $(q^{-i})$ are Kakeya sequences.

\begin{proposition}\label{p32}
If $1<q<2$ and 
\begin{equation}\label{32}
\frac{1+\inf_j\varepsilon_j}{1+\sup_j\varepsilon_j}\ge q-1,
\end{equation} 
then $\left(\frac{1}{q^i(1+\varepsilon_i)}\right)$ is a Kakeya sequence.
\end{proposition}

\begin{proof}
By our assumption we have $q^i(1+\varepsilon_i)>0$ for all $i$, and $1+\inf_j\varepsilon_j>0$.
Hence
\begin{equation*}
q^i(1+\varepsilon_i)\ge q^i(1+\inf_j\varepsilon_j)\to+\infty,
\end{equation*}
and therefore $\frac{1}{q^i(1+\varepsilon_i)}\to 0$.

It remains to show for all $n\ge 1$ the inequalities
\begin{equation*}
\frac{1}{q^n(1+\varepsilon_n)}\le\sum_{i=n+1}^{\infty}\frac{1}{q^i(1+\varepsilon_i)}.
\end{equation*}
They follow from the relations
\begin{equation*}
\frac{1}{q^n(1+\varepsilon_n)}
\le \frac{1}{q^n(1+\inf_j\varepsilon_j)}
\le\sum_{i=n+1}^{\infty}\frac{1}{q^i(1+\sup_j\varepsilon_j)}
\le\sum_{i=n+1}^{\infty}\frac{1}{q^i(1+\varepsilon_i)},
\end{equation*}
where the first and third inequalities ar obvious, while the middle inequality
is equivalent to
\begin{equation*}
\frac{1+\sup_j\varepsilon_j}{1+\inf_j\varepsilon_j}
\le \frac{1}{q-1},
\end{equation*}
i.e., to our assumption \eqref{32}.
\end{proof}

\begin{example}\label{e23}
By Binet's formula \eqref{12} the Fibonacci expansions are equivalent to the expansions \eqref{31} with $q=\varphi $ and 
\begin{equation*}
\varepsilon_i
=\frac{(-1)^{i+1}}{\varphi^{2i}}.
\end{equation*}
In this case we have
\begin{equation*}
\inf_j\varepsilon_j=\varepsilon_2=-\frac{1}{\varphi^4}
\quad\text{and}\quad 
\sup_j\varepsilon_j=\varepsilon_1=\frac{1}{\varphi^2}.
\end{equation*}
Hence
\begin{equation*}
\frac{1+\inf_j\varepsilon_j}{1+\sup_j\varepsilon_j}
=\frac{1-\frac{1}{\varphi^4}}{1+\frac{1}{\varphi^2}}
=\frac{\varphi^4-1}{\varphi^2(\varphi^2+1)}
=\frac{\varphi^2-1}{\varphi^2}
=\frac{\varphi^2-1}{\varphi+1}
=\varphi-1,
\end{equation*}
so that the condition of Proposition \ref{p32} is satisfied.
\end{example}

\section{Proof of Theorem \ref{t14}}\label{s4}

We need a new lemma.
Let $(p_i)$ be as in Theorem \ref{t14}.
We say that $p_n$ is a \emph{special element} if the condition \eqref{18} is satisfied.

\begin{lemma}\label{l41}
There exists a $K\ge 1$ such that if we remove  from $(p_i)$ a special element $p_k$ with $k>K$, then the remaining sequence still satisfies the corresponding hypotheses of Theorem \ref{t14}.
\end{lemma}

\begin{proof}
Fix an $N\ge 1$ such that 
\begin{equation*}
p_n\le 2p_{n+1}\quad\text{for all}\quad n\ge N,
\end{equation*}
and set
\begin{equation*}
\varepsilon:=\min\left\{-p_n+\sum_{i=n+1}^{\infty}p_i\ :\ n=1,\ldots,N\right\}.
\end{equation*}
By assumption \eqref{17} we have $\varepsilon>0$.
Since $p_i\to 0$, there exists a $K$ such that $p_i<\varepsilon$ for all $i\ge K$.

If we remove from $(p_i)$ a special element  $p_k$ with $k\ge K$, then the conditions \eqref{16}, \eqref{18} and \eqref{19} obviously remain valid.
The condition \eqref{17} also remains valid for all $n>k$ because the corresponding inequalities are unchanged, and it also remains valid for all $n\le N$ by the choice of $K$.
It remains to show that
\begin{equation}\label{41}
p_n+p_k<\sum_{i=n+1}^\infty p_i\quad\text{for all}\quad N<n<k.
\end{equation}
This is true for $n=k-1$  by \eqref{18} because $p_k$ is a special element.
Proceeding by induction, if \eqref{41} holds for some $N<n<k$, then it also holds for $n-1$.
Indeed, since $p_{n-1}\le 2p_n$ by our choice of $N$, applying  \eqref{41} we get
\begin{equation*}
p_{n-1}+p_k
\le 2p_n+p_k
<p_n+\sum_{i=n+1}^\infty p_i
=\sum_{i=n}^\infty p_i.\qedhere
\end{equation*}
\end{proof}

\begin{proof}[Proof of Theorem \ref{t14}]
Given $0<x<S$, using \eqref{16} we may apply repeatedly Lemma \ref{l41} to construct a sequence $p_{i_1}>p_{i_2}>\cdots$ of special elements such that 
\begin{equation}\label{42}
\sum_{j=1}^{\infty}p_{i_j}\le\min\{x,S-x\}.
\end{equation}
We may assume that $i_{j+1}>i_j+1$ for infinitely many indices $j$; then after the removal of the elements $p_{i_j}$ we still have an infinite sequence, that we denote by $(p_i')$.
Since after the removal of any finite number of elements $p_{i_1},\ldots, p_{i_m}$ the remaining sequence still satisfies the corresponding conditions \eqref{16} and \eqref{17} of Theorem \ref{t14}, letting $m\to\infty$ we conclude that $(p_i')$ is a Kakeya sequence.
Now we may obtain a continuum of expansions of $x$ as follows.
Fix an arbitrary sequence $(c_{i_j})\subset\{0,1\}$.
Then 
\begin{equation*}
0\le\sum_{j=1}^{\infty}c_{i_j}p_{i_j}
\le\min\{x,S-x\}
\end{equation*}
by \eqref{42}, so that 
\begin{equation*}
0\le x-\sum_{j=1}^{\infty}c_{i_j}p_{i_j}
\le x\le S-\sum_{j=1}^{\infty}p_{i_j}
=\sum_{i=1}^{\infty}p_i'.
\end{equation*}
By Theorem \ref{t11} there exists a sequence $(c_i')\subset\{0,1\}$ such that
\begin{equation*}
\sum_{i=1}^{\infty}c_i'p_i'=x-\sum_{j=1}^{\infty}c_{i_j}p_{i_j}.
\end{equation*}
Then 
\begin{equation*}
\sum_{j=1}^{\infty}c_{i_j}p_{i_j}+\sum_{i=1}^{\infty}c_i'p_i'
\end{equation*}
is an expansion of $x$ in the original system $(p_i)$ (the order of the positive terms is irrelevant), and different sequences $(c_{i_j})$ lead to different expansions.
\end{proof}

In order to state a corollary of Theorem \ref{t14} we generalize the geometric sequences.
Given a  real number $\rho>0$, a sequence $(p_i)$ of positive numbers is called a \emph{$\rho$-sequence} if 
\begin{equation*}
p_{i+1}\ge\rho\cdot p_i\quad\text{for all}\quad i.
\end{equation*}

\begin{corollary}\label{c42}
If $(p_i)$ is a $\rho$-sequence with $\rho>\frac{1}{\varphi}$, and $p_i\to 0$, then every $x\in (0,\sum_ip_i)$ has a continuum of  expansions
\begin{equation*}
x=\sum_{i=1}^{\infty}c_ip_i,\quad (c_i)\in\{0,1\}^{\mathbb N}.
\end{equation*}
\end{corollary}

\begin{proof}
It suffices to check the conditions of Theorem \ref{t14}.
We have $p_i\to 0$ by assumption, and 
\begin{equation*}
p_n\le \frac{1}{\rho}p_{n+1}<\varphi p_{n+1}<2p_{n+1}
\end{equation*}
for all $n$, so that the assumptions \eqref{16} and \eqref{19} are satisfied.
Furthermore, since $(p_i)$ is a $\rho$-sequence with $\rho>\frac{1}{\varphi}$, and since $\rho<1$ by our assumption $p_i\to 0$, the following relations hold for every $n\ge 2$:
\begin{equation*}
\sum_{i=n+1}^{\infty}p_i
\ge p_{n-1}\sum_{i=2}^{\infty}\rho^i
=\frac{\rho^2}{1-\rho}p_{n-1}
>p_{n-1}.
\end{equation*}
This proves \eqref{18} for all $n\ge 2$, and this implies \eqref{17}. 
\end{proof}

\begin{example}
Corollary \ref{c42} reduces to Theorem \ref{t13} if $p_i=q^{-i}$ with $q\in(1,\varphi)$.
\end{example}

\section{Proof of Theorem \ref{t16}}\label{s5}

For the proof of Theorem \ref{t16} we need some more properties of the Fibonacci sequence.
We recall the identity
\begin{equation}\label{51}
F_i^2=F_{i-1}F_{i+1}+(-1)^{i+1},\quad i=2,3,\ldots .
\end{equation}
It holds for $i=2$ by a direct inspection: $1^2=1\cdot 2-1$.
Proceeding by induction, if it holds for some $i\ge 2$, then  it also holds for $i+1$ because
\begin{align*}
&F_i^2=F_{i-1}F_{i+1}+(-1)^{i+1}\\
\Longrightarrow\ 
&F_i(F_i+F_{i+1})=(F_{i-1}+F_i)F_{i+1}+(-1)^{i+1}\\
\Longrightarrow\ 
&F_iF_{i+2}=F_{i+1}^2+(-1)^{i+1}\\
\Longrightarrow\ 
&F_{i+1}^2=F_iF_{i+2}+(-1)^{i+2}.
\end{align*}

\begin{lemma}\label{l51}
If $k$ is a positive odd integer, then
\begin{equation}\label{52}
\frac{1}{F_k}<\sum_{i=k+2}^{\infty}\frac{1}{F_i}.
\end{equation}
\end{lemma}

\begin{proof}
Fix any positive integer $k$, and set 
\begin{equation*}
\alpha:=\frac{F_{k+1}}{F_k}, \quad
\beta:=\frac{F_{k+2}}{F_{k+1}}, \quad
\gamma:=\max\{\alpha,\beta\}.
\end{equation*}
Then $F_{i+1}\le\gamma F_i$ for $i=k, k+1$, and hence by induction for all $i\ge k$. 
Since $\alpha\ne\beta$ by \eqref{51}, one of the equalities $F_{k+1}\le\gamma F_k$ and $F_{k+2}\le\gamma F_{k+1}$ is strict.
Therefore
\begin{equation*}
\sum_{i=k+2}^{\infty}\frac{1}{F_i}
>\frac{1}{F_{k+1}}\sum_{j=1}^{\infty}\frac{1}{\gamma^j}
=\frac{1}{F_{k+1}}\cdot\frac{1}{\gamma-1}=\frac{1}{F_k}\cdot\frac{1}{\alpha\cdot(\gamma-1)}.
\end{equation*}
If $k$ is odd, then this implies \eqref{52} because  $\beta>\alpha$ by \eqref{51}, so that $\gamma=\beta$, and hence
\begin{equation*}
\alpha\cdot(\gamma-1)=\frac{F_{k+1}}{F_k}\cdot \left(\frac{F_{k+2}}{F_{k+1}}-1\right)
=\frac{F_{k+2}-F_{k+1}}{F_k}
=1.\qedhere
\end{equation*}
\end{proof} 

\begin{proof}[Proof of Theorem \ref{t16}]
The conditions \eqref{16}, \eqref{17} et \eqref{19} of Theorem \ref{t14} are satisfied for $p_i:=1/F_i$ by Lemma \ref{l31}, and \eqref{18} is satisfied by Lemma \ref{l51}.
\end{proof}

\emph{Acknowledgment.}
The authors thank Mike Keane for suggesting the study of Fibonacci expansions.

\end{document}